\documentclass[11pt]{article}

\usepackage{amssymb,amsmath,amsthm}
\usepackage{geometry}
\usepackage[utf8]{inputenc} 

\newcommand{\Hyp}{\mathbb{H}}
\newcommand{\C}{\mathcal{C}}

\newcommand{\R}{\mathbb{R}}

\newcommand{\Z}{\mathbb{Z}}

\newcommand{\bS}{\mathbb{S}}

\newcommand{\cC}{\mathcal{C}}

\newcommand{\fa}{\mathfrak{a}}

\newcommand{\G}{\Gamma}

\newcommand{\tv}{\rightarrow}

\DeclareMathOperator{\PSL}{PSL}

\DeclareMathOperator{\Teich}{Teich}
\DeclareMathOperator{\Isom}{Isom}

\DeclareMathOperator{\Id}{Id}

\newtheorem{theorem}{Theorem}[section]
\newtheorem{definition}[theorem]{Definition}
\newtheorem{lemma}[theorem]{Lemma}

\DeclareMathOperator{\Hom}{Hom}
\author{Olivier Glorieux}

\title{The embedding of the space of negatively curved surfaces in geodesic currents}

\makeindex
\begin{document}
\maketitle

\begin{abstract}
We prove by an algebraic method that  the embedding of the Teichmüller space in the space of geodesic currents  is totally linearly independent.  We prove a similar result for all negatively curved surfaces using an ergodic argument.
\end{abstract}
\section{Introduction}
Let $S$ be a closed connected oriented surface of genus $g\geq 2$. By the  uniformization theorem, the Teichmüller space of $S$, $\Teich(S)$, identifies with the moduli space of marked hyperbolic structures. That is 
$\Teich(S)  = \{ (X,f_X)\}/\sim$
where $X $ is an hyperbolic surface and  $f_X \,: \,S \tv X $ is  an orientation preserving homeomorphism. The equivalence relation is given by $(X,f_X) \sim (Y,f_Y)$ if and only if there exists an isometry $g \, : \, X \tv Y$ such that $g\circ f_X $ is isotopic to $f_Y$. 

 The universal cover $\tilde{X}$ of a hyperbolic surface identifies with the hyperbolic plane $\Hyp^2$, and the marking induces a representation of $(f_X)_* \, : \, \pi_1(S) \tv \pi_1(X) < \Isom^+(\Hyp^2) \simeq \PSL_2(\R)$, well defined up to a conjugacy by an isometry of $\Hyp^2$. 
This give an embedding of the  Teichmüller space  into  the space of representations: $Rep:=\Hom(\pi_1(S) \tv \PSL_2(\R))/\PSL_2(\R).$
The Teichmüller space actually is one of the two connected components of $Rep$ which consists of only discrete and faithful representation. The other such component is $\Teich(\bar{S})$ the Teichmüller space of the surface with the opposite orientation. 

Our work concerns another embedding of the Teichmüller space in a very large vector space, called the space of \emph{geodesic currents} and denoted by $\cC(S)$. 
This embedding, say $L : \Teich(S)\tv \cC(S)$, has been introduced by F. Bonahon in \cite{bonahon1988geometry}, and this article aims to understand the linear properties of this embedding. We prove that this embedding is as linearly independent as possible, namely:
\begin{theorem}\label{th-main}
$\{L(S)\, |\, S \in \Teich(S)\}\subset \cC(S)$ is  a linearly independent set of vectors. 
\end{theorem}

The proof relies on algebraic arguments and decomposes into two parts. In Section \ref{sec-zariski}, we show that $n$ different surfaces induce via the direct sum representations of the holonomies, a Zariski dense 
subgroup of $\PSL_2(\R)^n$ and  in Section \ref{sec-independance} we use a famous theorem of Benoist to show the total linear independence of $L$. 

This proof can surely be extended to the embedding of the Hitchin component as presented by Martone--Zhang \cite{martone2019positive}, as soon as we know the different possible Zariski closures for these representations (it seems to be known and follows from a unpublished work of Guichard) 

However for non symmetric negatively curved manifolds this algebraic argument does not hold. P. Haissinsky suggested to us a nice ergodic argument that solves the surfaces case.  Let $M_{<0}(S)$ the space of negatively curved metrics on $S$ up to isotopy. In a similar fashion, one can associate to a metric $m$ its Liouville's current $L(m)$. We get the following theorem: 
\begin{theorem}\label{th-main neg curv}
$\{L(m)\, |\, (S,m) \in M_{<0}(S)\}\subset \cC(S)$ is  a linearly independent set of vectors. 
\end{theorem}
Except from the ergodic argument, this non algebraic proof relies on the marked length spectrum rigidity Theorem of Otal. More precisely, it uses the fact that the map $m\mapsto L(m)$ is injective. This is only known in dimension $2$ and the argument cannot be extended to any dimension.

\paragraph{Remark} Even if  Theorem  \ref{th-main neg curv} is more general than Theorem \ref{th-main}, we feel that the two approaches have their own interests and that is why we included the purely algebraic proof. For example, the side result on the Zariski closure of direct sum representations in the $\PSL_2(\R)^n$ can certainly be used and generalized in the study of discrete groups in higher rank symmetric spaces.

\paragraph{Organisation of the paper}
In the next section, we present the space of geodesic currents,  Bonahon's embedding of the Teichmüller space and the intersection function. Then we study the Zariski closure of the direct sum representations of holonomies for $n$ Teichmüller representations and prove Theorem \ref{th-main}. 
Finally, in the last section we prove Theorem \ref{th-main neg curv}. 

\paragraph{Acknowledgements}
This project received funding from the European Research Council (ERC) under the European Union's Horizon 2020 research and innovation programme (ERC starting grant DiGGeS, grant agreement No 715982. 

I am very grateful to Peter Haissinsky for its interest and its contribution to Theorem \ref{th-main neg curv} and I also want to thank Fanny Kassel and Jean-Philippe Burelle for fruitful discussions.

\section{Geodesic currents}

Recall that $S$ is a closed connected oriented surface of genus $g\geq 2$. We denote by $\G:=\pi_1(S)$ its fundamental group and $\tilde{S}$ the universal cover of $S$. 
The group $\G$ is a hyperbolic group and we denote by $\partial \G$ its Gromov boundary. If $S$ is endowed with a negatively curved metric, then one can identify the boundary of $\tilde{S}$ with $\bS^1$, homeomorphic to the boundary of $\G$. In this case, the geodesics of $\tilde{S}$ identify, via their endpoints to $\partial^{(2)} \bS^1 := \left(\partial \bS^1  \times \partial \bS^1\setminus \Delta\right)/ \Z/2\Z$, where $\Delta$ is the diagonal of the product and the action of $\Z/2\Z$ is given by the exchange of the two factors. 

\begin{definition}
The space of geodesic currents is the set of  $\G$ invariant measures  on $\partial^{(2)}\G \simeq\partial^{(2)}\bS^1 $. It will be denoted by $\cC(S)$. 
\end{definition}

This does not depend on the  identification of $\partial \G$ with $\partial \bS^1 $.

%
%
%
For the purpose of this article, we will not need a precise description of geodesic currents and will instead focus on two major examples. 


For the first example it will be more convenient to endow $S$ with a negatively curved metric. Then 
for any closed curve $c$ on $S$, there is a unique geodesic representative.  It lifts to a $\G$ invariant subset of geodesic on $\tilde{S}$, that we see as a subset of $\partial^{(2)} \bS^1$. The Dirac measure on each of this lift give a $\G$ invariant measure, that is a geodesic current. In the same manner, we can also see a (positive) linear combination of different closed curve as a geodesic current. 
We will follow the usual slight abuse of notations, and will not make the difference between the curve $c$  and its associated current. 

\begin{theorem}\cite{bonahon1988geometry}\label{th-density}
The set of linear combinations of closed curved is dense in  $\cC(S)$. 
\end{theorem}

The second major example is the Liouville current.  
The Liouville measure $L$ is  a  $\PSL_2(\R)$ invariant measure on $\partial^{(2)} \bS^1$, absolutely continuous to the Lebesgue measure $d\alpha d\beta$,  given by 
$$L = \frac{d\alpha d\beta}{|e^{i\alpha} - e^{i\beta} |^2}.$$

Now let $(S,m)$ be the surface $S$ endowed with a hyperbolic metric $m$. This gives an isometric diffeomorphism 
$\phi_m : \tilde{S} \tv \Hyp^2$, well defined up to composition by  an isometry of $\Hyp^2$. 
This defines a boundary homeomorphism  $ \partial \phi_m : \partial \tilde{S} \tv \partial  \Hyp^2$, well defined up composition by an element of $\PSL_2(\R)$. Since the Liouville measure on $\Hyp^2$ is invariant by $\PSL_2(\R)$ we can pull back $L$ on $\partial^{(2)}\tilde{S}\simeq \partial^{(2)}\bS^1$.  We denote it by $L_m$. 
We have the following: 
\begin{theorem}\cite{bonahon1988geometry}
The map $L: \Teich(S) \tv \cC(S)$ that sends $(S,m)$ to  $L_m$  is a proper, continuous, embedding. 
\end{theorem}

For any negatively curved metric $(S,m) \in M_{<0} (S)$ one can also pull back the Liouville measure and Otal showed:
\begin{theorem}\cite{otal1990spectre}\label{Th-otal}
The map $L: M_{<0}(S) \tv \cC(S)$ that sends $(S,m)$ to  $L_m$  is injective.
\end{theorem}

Bonahon also introduced a bilinear form on $\cC(S)$, the so called intersection function, denoted by $i(\cdot, \cdot) : \cC(S)\times \cC(S) \tv \R$. It generalizes the notion of geometric intersection between two closed curves. Although the definition is a bit technical, the only property we will need will be:
\begin{theorem}
Let $L_m$ be a Liouville current associated to $(S,m)$, let $c$ be a closed curve on $S$. Then 
$$i(L_m,c) =\ell_m(c),$$
where  $\ell_m(c)$ is the length on $(S,m)$ of the unique geodesic representative of $c$. 
\end{theorem}

\section{The embedding of the Teichmüller space}
\subsection{Zariski closure of diagonal representation}\label{sec-zariski}
We will need the following result, known as Goursat's Lemma : 

\begin{theorem}[Goursat's Lemma]
Let $G,G'$ be two groups. Let $ p_1 :  H \tv G $, $p_2  : H \tv G'$ be two surjective homomorphisms. Identify  $N'$ the kernel of $p_1$ with a normal subgroup of $G'$ and $N$, the kernel of $p_2$ with a normal subgroup of $G$. Then the image of $H$ in $G/N\times G'/N'$ is the graph of an isomorphism 
$G/N\simeq G'/N'$. 
\end{theorem}

Recall that to a hyperbolic surface $(S,m)$, the holonomy map gives a discrete and faithful representation  $\rho_m \, :\, \G \tv \PSL_2(\R)$, well defined up to conjugacy by $\PSL_2(\R)$.
The outer automorphism $\tau$ of $\PSL_2(\R)$ is given by the conjugation with $\left(\begin{array}{cc} 1 & 0\\
0 & -1
\end{array}\right)$
and gives the corresponding hyperbolic structure on the surface with the opposite orientation. 

Finally recall that any discrete and faithful representation of a compact surface fundamental group into $PSL_2(\R)$ is Zariski dense. 

\begin{theorem}\label{Th-Zariski density}
Let $S$ be a closed connected oriented surface of genus $g\geq 2$. Let $X_1, ... , X_n$ be $n$  hyperbolic surfaces and let $\rho_i$ be the corresponding holonomies. Let  $\rho=(\rho_1, ..., \rho_n) :\G \tv \PSL_2(\R)^n$ be the direct sum representation. Then $\rho(\G)$ is  Zariski dense if and only if all surfaces are different in $\Teich(S)$. 
\end{theorem}

\begin{proof}
If two surfaces,  say $X_1, X_2$, correspond to the same point in the Teichmuller space, then there exists $h\in \PSL_2(\R)$ such that their holonomies satisfy $\rho_1 =h \rho_2 h^{-1} $. Therefore the Zariski closure of $\rho$ is contained in $H \times \PSL_2(\R)^{n-2}$, where $H$ is the diagonal copy of $\PSL_2(\R)$ given by $H=\{(g, hgh^{-1})\, |\, g\in \PSL_2(\R)\}. $  

We will prove the converse by induction. 

For $n=1$ the result is trivial. Suppose that the result is true for $n$ and consider $(n+1)$ non isometric hyperbolic structures on $S$ : $X_1, ..., X_{n+1}$. Let $\rho=(\rho_1, ... , \rho_{n+1})$ be the diagonal representation corresponding to their holonomies. Let $H$ be the Zariski closure of $\rho(\G)$.  Consider the projection on the last factor $\pi_{n+1} : H \tv \{\Id\}^n\times \PSL_2(\R) $ and the one on the first $n$ factors: $p : H \tv \PSL_2(\R)^n \times \{ \Id\}$. 

By induction the projection $p$ is surjective and by hypothesis $\pi_{n+1}$ also is. Since $\PSL_2(\R)$ is simple, the kernel of $p$ can  either be $\PSL_2(\R)$ or $\{\Id\}$. In the first case, by Goursat's lemma, the kernel of $\pi_{n+1}$, $N$, has to verify $\PSL_2(\R)^n/N \simeq \PSL_2(\R)/\PSL_2(\R)$ and therefore $H=\PSL_2(\R)^{n+1}$.

In the second case, the kernel of $\pi_{n+1}$, $N$, has to verify $\PSL_2(\R)^n/N \simeq \PSL_2(\R)$. Since $\PSL_2(\R)$ is simple, we claim that the only possibility  for $N$ is $N=\PSL_2(\R)^{k} \times \{\Id\} \times \PSL_2(\R)^{n-k-1}$ for some $k\in \{0, ..., n-1\}$, see Lemma \ref{lem-simple}. Then by Goursat's Lemma, the projection of $H$ in $\PSL_2(\R)\times \Big(\PSL_2(\R)^n/N\Big) $ identifies with the graph of an isomorphism  from  $\PSL_2(\R)$ to $ \Big(\{ \Id\}^{k}\times \PSL_2(\R)\times \{\Id \} ^{n-k-1}\Big)\simeq \PSL_2(\R)$, and finally implies that $\theta(\rho_1) =\rho_k$ for an  automorphism of $\PSL_2(\R)$. 

Since the surface is  oriented, $\theta$ has to be a inner automorphism, and therefore $\rho_1$ is conjugated to $\rho_k$. This is a contradiction.

\end{proof}

We prove the claim from the second part of the proof
\begin{lemma}\label{lem-simple}
Let $N$ be a normal subgroup of $\PSL_2(\R)^n$, such that $\PSL_2(\R)^n /N\simeq \PSL_2(\R)$. Then there exists $k\in \{0, ..., n-1\}$ such that $N=\PSL_2(\R)^{k} \times \{\Id\} \times \PSL_2(\R)^{n-k-1}$.
\end{lemma}

\begin{proof}
For the sake of clarity we will denote $G_i$ the $i$-th $\PSL_2(\R)$ factor in $\PSL_2(\R)^n$ and  $G$ will designated any abstract $\PSL_2(\R)$. 
Consider $p_i : G^n \tv G_i$. Then  $p_i(N)$ is a normal subgroup of $G_i$. By simplicity $p_i(N)$ is either trivial or  the whole group $G_i$. Suppose that there exists $i\neq j$ such that $p_i(N) = p_j(N) = \{\Id\}$, and by symmetry with suppose that $(i,j)=(1,2)$ then 
$N \subset \{\Id\} \times \{\Id \} \times G^{n-2}$ and therefore $G^n/N \supset G^2 $ which is absurd. 
Consequently there is at most one projection such that the image of $N$ is trivial. 

Consider two projections $p_i,p_j $  with non trivial image. Denote by $N_{ij} := N \cap (G_i\times G_j)$. 
Remark that $N_{ij}$ is normal in $G_i\times G_j$. Then applying Goursat lemma, one see that either
$N_{ij} =G_i \times G_j$ or $N_{ij}$ is the graph of a homomorphism between $G_i$ and $G_j$. The latter cannot happen since it would not be a normal subgroup. 
Therefore $N_{ij}=G_i\times G_j$. 

Applying this argument for all pairs with non trivial image, we see that $N = \Pi_{k} G_k$ where the product if taken over all $k$ such that $p_k(G)\neq \Id$. Since $G^n/N \simeq G$, there is at least one factor such that $p_k(G) \neq \Id$, and finally exactly one by the previous argument.  
\end{proof}

\paragraph{Remark:} This lemma generalizes to any non-abelian product of simple groups.

\subsection{Total independance of Liouville's current}\label{sec-independance}

We are going to give Benoist's theorem \cite{benoist1997proprietes} in the context of  $\PSL_2(\R)^n$. Any hyperbolic element $g\in \PSL_2(\R)$ is conjugated to an element of the form $\left( \begin{array}{cc}
e^{\lambda(g)/2} & 0\\
0 & e^{-\lambda(g)/2}
\end{array}\right)$. The number $\lambda(g)\in \R^+$ is the translation length of $g$.

Let  $g=(g_1, ..., g_n)\in \PSL_2(\R)^n$ be a loxodromic element, that is all $g_i$ are  hyperbolic elements. The vector $\lambda(g):=(\lambda(g_1),..., \lambda(g_n) ) \in (\R^+)^n$ is called the Jordan projection of $g$

\begin{definition}
Let $H$ be a subgroup of $\PSL_2(\R)^n$.  The limit cone of $H$ is defined by 
$$C(H) :=\overline{\cup_{h\in H_{lox}} \lambda(h) \R^+} \subset(\R^+)^n.$$
\end{definition}

%
%
%
%


One of the striking properties of the limit cone is given by the following theorem of Y. Benoist: 
\begin{theorem}\cite{benoist1997proprietes}\label{Benoist}
 If $H$ is Zariski dense in $\PSL_2(\R)^n$, then its limit cone  has non-empty interior. 
\end{theorem}

We are now going to show the main result of this section.

\begin{theorem}
$\{L_m \, |\, (S,m)\in \Teich(S)\}$ is a linearly independent family of geodesic currents. 
\end{theorem}

\begin{proof}
Let $(S_1, ..., S_n)$ a finite family of distinct hyperbolic surfaces, and denote by $L_k$ the corresponding Liouville currents. Suppose that there exists $(a_1, ... , a_n ) \in \R^n$ a family of real number such that 
$$\sum_{k=1}^n a_kL_k =0.$$
We have therefore, for all closed curve $c \in \C$: 
\begin{equation}\label{eq-relation on length}
\sum_{k=1}^n a_ki (L_k , c) = \sum_{k=1}^n a_k\ell_{S_k}(c) =0.
\end{equation}

Consider the holonomy representations $\rho_k : \pi_1(S) \tv \PSL_2(\R)$ and let $\rho = (\rho_1, ..., \rho_n)$ be the diagonal representation of $\G \tv \PSL_2(\R)^n$. 

By Theorem \ref{Th-Zariski density}, since  $(S_k)$ are $n$ distinct hyperbolic surfaces  $\rho(\G)$ is  Zariski dense. 
Therefore by Benoist's Theorem \ref{Benoist}, the limit cone of $\rho(\G)$, $\cC(\rho(\G)) \subset \fa^+ $ is of non-empty interior 

However, by Equation (\ref{eq-relation on length}), it is also contained in the kernel of the linear form  $\ell:  \R^n \tv \R$, $\ell(x_1, ..., x_n) = \sum_k a_kx_k$. A proper vector subspace of $\R^n$ has  empty interior. Therefore, the linear form $\ell$ is the zero map, and $(a_1, ..., a_n ) =(0,... , 0)$. 
\end{proof}

\section{Embedding of negatively curved metric}
We prove in this section the more general result that the space of negatively curved metric on $S$ with volume $1$ embedded totally independently in the space of geodesic current via Liouville's current. 

The result follows from an ergodic argument which has been suggested to us by P. Haissinky. 

Let $\mu$ be a geodesic current. Considering the measure $m = \mu \mathrm{d} t $ on  $\mathbb{S}^{(2)} \times \R$ we get a measure on the unitary tangent bundle of $T^1\tilde{S}$ invariant by the geodesic flow and by $\G$. The set of geodesic currents identify by this procedure to the measure on  $T^1 S$ invariant by the geodesic flow. From Liouville current we get the Liouville measure on the unitary tangent bundle of $S$. 

A classical result due to Anosov tells us that  in the case of negatively curved metric the Liouville measure is ergodic, see for example \cite[Theorem 10.11]{pesin}.

Let $X$ be a compact metric space and $\phi_t$ an action on  $X$. As a consequence of Birkhoff's ergodic theorem, for any ergodic measure $\mu$ on $X$, there exists a set $\Omega_\mu$  of full $\mu$-measure, called the generic points of $\mu$, satisfying that for every continuous function $f$ on $X$ and every point $x\in \Omega_\mu$ one has 
$$\lim_{T \tv \infty} \frac{1}{T}\int_0^T  f(\phi_t(x)) d t = \int_X f d\mu.$$
Therefore,  if we have $n$ ergodic probability measures $(\mu_1, ..., \mu_n)$ on $X$,  the set $\Omega_{\mu_i}$, for $i \in \{1,..., n\}$ are disjoint (in the measure theoretical sense). 

In particular, if we consider $n$ distinct Liouville currents $(L_1, ...., L_n)$ that we see as invariant measure on $T^1S$, we can associate  $n$ subset $\Omega_i $ on $T^1S$  such that 
$L_k(\Omega_l)=0$ for all $l\neq k$ and $L_k(\Omega_k) =1$. 

From Otal's Theorem \ref{Th-otal}, the map that associates the Liouville current to a negatively curved metric on $S$ is injective. Therefore if one has $n$ negatively curved metrics, one gets $n$ distinct Liouville currents. 

Now the proof of Theorem \ref{th-main neg curv} is straightforward. 
Let $a_1, ... , a_n\in \R^n$ be $n$ real numbers such that $\sum_k a_k L_k=0$, evaluate this equality on $\Omega_j$, this gives 
$a_j L_j(\Omega_j) = a_j =0$, and conclude the proof.

\bibliographystyle{alpha}

\end{document}